\documentclass{amsart}
\usepackage{color}

\newtheorem{thm}{Theorem}
\newtheorem{lem}{Lemma}[section]

\theoremstyle{definition}
\newtheorem{definition}[thm]{Definition}
\newtheorem{example}{Example}

\newtheorem{prop}{Proposition}[section]

\newtheorem{theorem}{Theorem}[section]
\newtheorem{lemma}{Lemma}%

\newcommand{\blem}{\begin{lemma}}
\newcommand{\elem}{\end{lemma}}
\newcommand{\bexer}{\begin{exe}}
\newcommand{\eexer}{\end{exe}}
\newcommand{\beq}{\begin{eqnarray}}
\newcommand{\eeq}{\end{eqnarray}}
\newcommand{\bthm}{\begin{theorem}}
\newcommand{\ethm}{\end{theorem}}

\newcommand{\bdefe}{\begin{definition}}

\newcommand{\edefe}{\end{definition}}

\newcommand{\bprop}{\begin{prop}}
\newcommand{\eprop}{\end{prop}}

\newcommand{\bpf}{\begin{proof}}
\newcommand{\epf}{\end{proof}}

\def\be{\begin{equation}}
\def\ee{\end{equation}}

\newcommand\Hyp{\mathrm{Hyp}}
\newcommand\Kob{\mathrm{Kob}}

\numberwithin{equation}{section}

\newcommand{\D}{\mathbb{D}}
\input amssym.def
\input amssym

\begin{document}

\author{M. Mateljevi\' c}
\address{Faculty of mathematics, University of Belgrade, Studentski Trg 16,
Belgrade, Yugoslavia} \email{miodrag@matf.bg.ac.rs}
\title[ Schwarz lemma] {Schwarz lemma and      Kobayashi Metrics for holomorphic and pluriharmonic functions}

\date{28 Dec,2016}


\maketitle
This is a working version. We will polish the text  as soon as possible.

\section{Introduction}

The Schwarz lemma as one of the most influential results in complex analysis and it has
a great impact to the development of several research fields, such as geometric function
theory, hyperbolic geometry, complex dynamical systems,
and theory of quasi-conformal mappings.
In  this note  we mainly  consider   various version  of Schwarz lemma and its relatives related to holomorphic  functions including several variables.

We use Ahlfors-Schwarz lemma
to give a  simple approach to Kalaj-Vuorinen results \cite{kavu} (shortly KV-results) and to put it  into a broader perspective.
But,  it turns out that our methods (results)   unify  very recent approaches   by D. Kalaj- M. Vuorinen, H. Chen, K. Dyakanov, D. Kalaj, M. Markovi\'c,  A. Khalfallah and P.Melentijevi\'c.\footnote{After writing  a version  of manuscript  \cite{MMSchw_Kob},   Petar Melentijevi\'c \cite{P_Mel} sent me his preprint (at 18 Jan 2017).
We learned from his  preprint about   a few related recent  results  \cite{hhChen,Dyak_Bloch,Kal_sch2,Mar_hyp}; we also found \cite{Khal} which may be related to our paper.}
We use several dimension version of Schwarz lemma (we call it the  geometric form  of  Kobayashi-Schwarz lemma, Theorem \ref{th:kob0}) to generalize these results to  several variables.
In particular our considerations include    domains on which we can compute  Kobayashi  distance, as   the unit ball, the polydisc,   the punctured disk  and the strip.
There is  a huge literature related to Schwarz  lemma (see for example H. Boas  \cite{Hboas}, R. Osserman  \cite{Oss}, D.M. Burns and S.G. Krantz \cite{BKr}, S.G. Krantz \cite{krantz3} and the literature cited there) and we apology if we did not mention some important papers.


\section{Schwarz lemma,background }\label{comp_geo}

Concerning material exposed in this section, we refer the interested reader to  well written    Krantz's paper   \cite{krantz2} that can provide greater detail.
Let  $D$ be  a domain in $z=x+iy$-plane and a
Riemannian metric be given by the fundamental form

\[ds^{2}=\sigma |dz|^2= \sigma(dx^{2}+dy^{2}) \]
which    is conformal with euclidian metric.
Often in the literature a Riemannian metric is given by $ds=\rho
|dz|$, $\rho > 0 $, that is
by the fundamental form
\[ds^{2}=\rho^{2}(dx^{2}+dy^{2})\,. \]
In some situations it is   convenient to   call  $\rho=\rho_D$ shortly
metric density.

For  $\mathbf{v}\in T_z \mathbb{C} $  vector, we  define     $|\mathbf{v}|_\rho=\rho(z) |\mathbf{v}|$.
If  $\gamma$ is  a  piecewise  smooth path in $D$, we    define   $|\gamma|_\rho= \int_{\gamma} \rho(z) |dz| $  and   $d_D(z_1,z_2)= \inf|\gamma|_\rho$, where the infimum is
taken over all paths $\gamma$ in  $D$   joining  the points
$z_1$ and   $z_2$.

For a hyperbolic plane domain $D$, we denote   respectively by $\lambda=\lambda_D$ (or if we wish to be more specific  by ${\rm  Hyp}_D$) and
$\delta_D$ (in some papers we use also notation $\sigma_D$)   the hyperbolic and pseudo-hyperbolic metric  on $D$
respectively.  By ${\rm  Hyp}_D(z)$ we also  denote the  hyperbolic density at $z\in D$.

For planar domains $G$ and $D$  we denote  by    ${\rm  Hol}(G,D)$  the class of all holomorphic mapping from $G$  into    $D$. For   complex Banach manifold   $X$  and   $Y$  we denote  by   $\mathcal{O}(X,Y)$  the class of all holomorphic mapping from $X$  into    $Y$.

If $f$ is a function on a set $X$ and   $x\in X $ sometimes we write   $fx$ instead of $f(x)$.
We write  $z=(z_1,z_2,...,z_n) \in  \mathbb{C}^n$.
On $\mathbb{C}^n$ we define the standard   Hermitian inner product  by
$<z,w>= \sum_{k=1}^n z_k \overline{w_k}$  for $z,w\in  \mathbb{C}^n$  and    by  $|z|= \sqrt{<z,z>}$ we denote the norm of vector $z$. We also use  notation  $(z,w)$ instead of  $<z,w>$ on some places.
By $\mathbb{B}=\mathbb{B}_n$ we denote the unit ball in  $\mathbb{C}^n$. In particular we use   also notation  $\mathbb{U}$  and  $\mathbb{D}$ for the unit disk in complex plane.
\bprop[classical Schwarz lemma 1-the unit disk]\label{:clschw}
Suppose that $f:\mathbb{D}\rightarrow \mathbb{D}$ is an analytic
map and  $ f(0)=0$. The classic Schwarz lemma states : $|f(z)|\leq
|z|$ and  $ |f'(0)| \leq 1$.
\eprop
It is interesting that this (at a first glance) simple result has far reaching applications and forms.

Define  $$T_{z_1}(z)= \frac{z-z_1}{1-\overline{z_1} z},$$    $\varphi_{z_1}
= -T_{z_1}$  and
$$\delta (z_1,z_2)=|T_{z_1}(z_2)| =|\frac{z-z_1}{1-\overline{z_1} z}|.$$

The classical Schwarz  lema  yields motivation to introduce hyperbolic  distance: If  $f\in {\rm  Hol} (\mathbb{U},\mathbb{U})$, then   $\delta(fz_1,fz_2)\leq  \delta (z_1,z_2)$.

Consider $F=\varphi_{w_1} \circ f \circ\varphi_{z_1}$, $w_k=f(z_k)$. Then  $F(0)=0$  and  $|\varphi_{w_1}(w_2)| \leq |\varphi_{z_1}(z_2)|$.

Hence
\begin{equation}\label{eq:clschw2}
|f'(z)| \leq \frac{1- |fz|^2}{1- |z|^2} .
\end{equation}
By  the  notation    $w=f(z)$ and  $dw=f'(z)dz$, we can rewrite   (\ref{eq:clschw2}) in the form

\begin{equation}\label{eq:clschw3}\frac{|dw|}{1- |w|^2} \leq \frac{|dz|}{1- |z|^2} .
\end{equation}
We can rewrite this inequality in vector form. Namely, define   the density   $\lambda(z)= \frac{1}{1- |z|^2}$.
For  $\mathbf{v}\in T_z \mathbb{C} $  vector we define     $|\mathbf{v}|_\lambda=\lambda(z) |\mathbf{v}|$
and set
$\mathbf{v}^*= df_z( \mathbf{v})$.  Hence, we can rewrite  (\ref{eq:clschw3}) in the form:   $|\mathbf{v}^*|_\lambda \leq |\mathbf{v}|_\lambda$.
Thus we have
\bprop[classical Schwarz lemma 2-the unit disk]\label{:clschw}
Suppose that $f:\mathbb{D}\rightarrow \mathbb{D}$ is an analytic
map. \\
(a) Then  (\ref{eq:clschw2}).\\
(b) If   $\mathbf{v}\in T_z \mathbb{C} $ and $\mathbf{v}^*= df_z( \mathbf{v})$,  then   $|\mathbf{v}^*|_\lambda \leq |\mathbf{v}|_\lambda$.
\eprop

If we choose  that  the hyperbolic density (metric) is given by
\begin{equation}\label{eq:hyp0}
{\rm  Hyp}_{\D}(z)=\frac{2}{1- |z|^2}\,,\quad z \in \D\,.
\end{equation}
then  the Gaussian curvature of  this metric is  $-1$.

We summarize
\beq\label{hyp.diUH}
\lambda_\mathbb{U} =\ln \frac{1+ \delta_\mathbb{U} }{1- \delta_\mathbb{U} }, \quad \lambda_\mathbb{H} =\ln \frac{1+ \delta_\mathbb{H}}{1- \delta_\mathbb{H}}\,.
\eeq



Let $G$ be a simply connected domain different from $\mathbb{C}$  and let $\phi:G \rightarrow \mathbb{U}$  be a conformal isomorphism. Define
$\varphi^G_a(z)= \varphi_b (\phi(z))$, where $b=\phi(a) $, and  the pseudo hyperbolic distance on $G$ by

 $\delta_G(a,z)=  |\varphi^G_a(z)|=  \delta_\mathbb{U}(\phi(a),\phi(z))$.
One can  verify that the pseudo hyperbolic distance on $G$ is independent of conformal mapping $\phi$.
In particular,   using conformal  isomorphism   $A(w)= A_{w_0}(w)= \frac{w-w_0}{w-\overline{w_0}}$ of $\mathbb{H}$ onto $\mathbb{U}$, we find  $\varphi_{H,w_0}(w)=A(w)$ and therefore
$\delta_H(w,w_0)= |A(w)|$.


For a domain $G$  in $\mathbb{C}$ and $z,z' \in G$ we define  $\underline{\delta}_G(z,z')=\sup  \delta_\mathbb{U}(\phi(z'),\phi(z)),$  and the Carath\'{e}odory
distance $c_G(z,z')=\sup  \delta_\mathbb{U}(\phi(z'),\phi(z)),$  where the suprimum is taken over all
 $\phi \in {\rm Hol}(G,\mathbb{U})$.
 Of course the Caratheodory distance can be trivial-for instance   if $G$   is the entire plane.

 On simply connected domains,   the pseudo-hyperbolic   distance  and  the hyperbolic   distance are related by
$$\delta= \tanh(\lambda/2)$$
and we have useful relation:\\
(I-1)   If  $G$ and $D$ are simply connected domains  domains   and  $f$  conformal mapping of $D$ onto $G$, then  ${\rm  Hyp}_G(fz)|f'(z)|={\rm  Hyp}_D(z)$.

The uniformization theorem says that every simply connected Riemann surface is conformally equivalent to one of the three Riemann surfaces: the open unit disk, the complex plane, or the Riemann sphere. In particular it implies that every Riemann surface admits a Riemannian metric of constant curvature. Every Riemann surface is the quotient of a free, proper and holomorphic action of a discrete group on its universal covering and this universal covering is holomorphically isomorphic (one also says: "conformally equivalent" or "biholomorphic") to one of the following:the Riemann sphere,the complex plane and   the unit disk in the complex plane.If the universal covering of a Riemann surface $S$  is  the unit disk we say that $S$ is hyperbolic.Using  holomorphic covering $\pi: \mathbb{U} \rightarrow S$, one can define the pseudo-hyperbolic  and  the hyperbolic metric on $S$.
In particular,  if  $S=G$   is hyperbolic planar domain  we can use\\
(I-2)    ${\rm  Hyp}_G(\pi z)|\pi'(z)|={\rm  Hyp}_{\D}(z)$    and \\
(I-3)   If  $G$ and $D$ are  hyperbolic domains   and  $f$  conformal mapping of $D$ onto $G$, then  ${\rm  Hyp}_G(fz)|f'(z)|={\rm  Hyp}_D(z)$.

\bprop [Schwarz lemma 1- planar hyperbolic domains]
(a)  If  $G$ and $D$ are conformally isomorphic to  $\mathbb{U}$   and   $f\in {\rm  Hol}(G,D)$, then
$$\delta_D(fz,fz')\leq \delta_G(z,z'), \quad z,z'\in G .$$

(b)  The result holds  more generally:  if  $G$ and $D$ are  hyperbolic domains  and  $f\in {\rm  Hol}(G,D)$, then
$$ {\rm  Hyp}_D(fz,fz')\leq {\rm  Hyp}_G(z,z'), \quad z,z'\in G . $$

(c) If  $z\in G$,  $\mathbf{v}\in T_z \mathbb{C} $ and   $\mathbf{v}^*= df_z( \mathbf{v})$,  then  $$|\mathbf{v}^*|_{{\rm  Hyp}} \leq |\mathbf{v}|_{{\rm  Hyp}}   .$$
\eprop
For a   hyperbolic planar domain $G$  the Carath\'{e}odory  distance  $C_G\leq \lambda_G$  with equality if and only if  $G$ is a simply connected domain.

(A)   holomorphic functions do not increase  the  corresponding hyperbolic distances between  the  corresponding hyperbolic domains.

The Caratheodory and Kobayashi metrics  have proved
to be important tools in the function theory of several complex
variables.   We can express   Kobayashi-Schwarz lemma in the geometric form (see Theorem \ref{th:kob0}, which is our main tool):
In particular,  we have:\\

(B) If $G_1$   and  $G_2$ are domains in $\mathbb{C}^n$  and $f: G_1\rightarrow G_2 $  holomorphic function,  then
$f$ does not increase  the  corresponding Caratheodory(Kobayashi)  distances.

But they are less familiar in the context of one complex
variable. Krantz \cite{krantz}
gathers in one place the basic
ideas  about these important invariant metrics for domains in the plane
and  provides some illuminating examples and applications.
We consider various generalization of this result including several variables.
There is an interesting connection between  hyperbolic geometry and  complex geometry.

(C) Pseudo-distances defined by pluriharmonic functions.

In \cite{Khal}  the author   constructs   $\alpha_{M,P}$, a new holomorphically invariant
pseudo-distance on a complex Banach manifold  $M$  using the set of real pluriharmonic functions on
$M$ with values in  $P$, a proper open interval of  $\mathbb{R}$.
It is well known that the Kobayashi pseudo-distance is the largest and
the Caratheodory pseudo-distance is the smallest one which can be assigned
to complex Banach manifolds by a Schwarz-Pick system. Therefore   $C_M\leq \alpha_{M,P}\leq \Kob_M$.

\section{Schwarz lemma for real harmonic  functions}\label{comp_geo}

Let  $\mathbb{S}= \mathbb{S}_0= \{w : |Re w|< 1\}$   and   $\mathbb{S}_1=  \{w : |Re w|< \pi/4\}$.
$\tan$ maps $\mathbb{S}_1$ onto  ${\mathbb D}$.  Let  $B(w)= \frac{\pi}{4} w$    and  $f_0=\tan \circ B $, ie.   $f_0(w)= \tan(\frac{\pi}{4} w) $. Then  $f_0$   maps $\mathbb{S}_0$ onto  ${\mathbb D}$.

Let  $r<1$,    $A_0(z)=\frac{1+z}{1-z}$,  and   let $\displaystyle \phi = i \frac{2}{\pi} {\mathrm {ln}} A_0;$.
 that is
$ \phi= \phi_0\circ A_0 $,  where
$\phi_0 = i \frac{2}{\pi} {\mathrm {ln}}$.
  Let $\hat{\phi}$ be defined by   $\hat{\phi} (z)=-\phi (iz)$. Note that    $\hat{\phi} =\frac{4}{\pi} \arctan $ is the inverse of $f_0$.
Hence  $\hat{\phi}'(0)=  \frac{4}{\pi}$ and  if
 $f$ is a conformal   map of $\D$ onto $\mathbb{S}_0$  with  $f(0)=0$, then  $|f'(0)= \frac{4}{\pi}$.

By the subordination principle,  \\
(II) If $f$ is an analytic  map of $\D$ into $\mathbb{S}_0$  with  $f(0)=0$, then  $|f'(0)\leq \frac{4}{\pi}$.

There is analogy of classical Schwarz lemma for harmonic maps.
\begin{lemma}[\cite{dur}]\label{Sc.Lemma1}
Let  $h: {\mathbb D}
\rightarrow \mathbb{S}_0$ be a harmonic mapping with $h(0)=0\,.$ Then
$|Re\, h(z)| \le \frac{4}{\pi} {\mathrm{tan}}^{-1}|z|$ and this
inequality is sharp for each point $z \in {\mathbb D} \,.$
\end{lemma}

\begin{example}
 Let   $f_a(z)= \rm{Re} \hat{\phi}(z)  +i ay$  and   $g_a(z)= \rm{Re} \hat{\phi}(z)  +i a\rm{Im} \hat{\phi}(z)$. For  $a\in \mathbb{R}$, $f_a$ and  $g_a$ are  harmonic maps of $\D$ into $\mathbb{S}_0$  and $f_a(0)=g_a(0)=0$.
 Since   $|f_a'(0)|\geq |a|$  and  $L_{g_a}(0)=|a|\frac{4}{\pi}  $, there is no  reasonable estimate for  the distortion of harmonic  functions which maps  the unit disk into the strip.
It is interesting to note that  Lemma  \ref{Sc.Lemma1} shows that we can control the growth of the real part of  a harmonic mapping which maps   $\D$ into $\mathbb{S}_0$ and keeps the origin fixed.
\end{example}
 Let   $\lambda_0$  be   a  hyperbolic density on  $\mathbb{S}_0$. Then
\begin{equation}\label{eq:dhyp-strip0a}
    \lambda_{0}(w)= {\rm  Hyp}_{\mathbb{S}_0}(z)= \frac{\pi}{2}  \frac{1}{\cos (\frac{\pi}{2} u)}.
\end{equation}

If  $F$ is holomorphic map from  $\D$ into   $\mathbb{S}_0$, then by Ahlfors-Schwarz lemma

\begin{equation}\label{eq:Ah-Sch1a}
    \lambda_0 (F(z)) |F'(z)|\leq  2 (1- |z|^2)^{-1},\quad z\in \D.
\end{equation}
Using that $1-\cos (\frac{\pi}{2}x)= 2 \sin^2 (\frac{\pi}{4}x)$  and  the inequality   $ \frac{4}{\pi}t \leq \sqrt{2} \sin t$, $0\leq  t \leq \frac{\pi}{4}$, we prove
$\cos (\frac{\pi}{2}x) \leq   1-x^2$,  $|x|\leq 1$, and therefore
we get\\
(A1)     $ \frac{\pi}{2} (1-u^2)^{-1} \leq \lambda_{0}(w)$.
Hence, we get Kalaj-Vuorinen \cite{kavu}, see also \cite{hhChen}:
\begin{prop}
If  $f$ is harmonic map from $\D$  into $(-1,1)$ with $f(0)=0$, then
\begin{eqnarray}
  |\nabla f(z)| \leq \frac{4}{\pi}  \cos (\frac{\pi}{2}f(z)) (1- |z|^2)^{-1}, \quad z\in \D. \label{eq:dhyp-strip1a}\\
    |\nabla f (z) | \leq  \frac{4}{\pi} \frac{1- |f(z)|^2 }{1-|z|^2},   \quad |z|< 1. \label{eq:dhyp-disk1a}
\end{eqnarray}
\end{prop}
\begin{proof}[Outline]
Let  $F$  be analytic  such that  $\Re f = \Re F$ on $\mathbb{D}$ with $F(0)=0$.  Then  $|\nabla f(z)|\leq |F'(z)|$, $z\in \D$.
By  Ahlfors-Schwarz   estimate  (\ref{eq:Ah-Sch1a}),
\[
    \lambda_0(f(z)) |\nabla f(z)| \leq  2 (1- |z|^2)^{-1}, \quad z\in \D.
\]
 Now an application of  the  formula  (\ref{eq:dhyp-strip0a})    yields  (\ref{eq:dhyp-strip1a}). By   (A1),   (\ref{eq:dhyp-disk1a}) follows from (\ref{eq:dhyp-strip1a}).
\end{proof}
By   (A1),   $\Hyp_{\mathbb{D}}(x_1,x_2) \leq \frac{4}{\pi} \Hyp_{S_0}(x_1,x_2)$,   $x_1,x_2 \in (-1,1)$.
If  $F$ is holomorphic map from  $\D$ into   $\mathbb{S}_0$  and $u=Re F$, then
$\Hyp_{S_0}(uz_1,u z_2)\leq \Hyp_{S_0}(Fz_1,F z_2)\leq  \Hyp_{\mathbb{D}}(z_1,z_2) $.

\begin{thm}[]\label{th:kavugen}
Suppose that  $D$ is a hyperbolic plane domain  $v : D \rightarrow (-1, 1)$ is real harmonic on
hyperbolic domain $D$. Then
$\Hyp_{\mathbb{D}}(v(z_1),v(z_2)\leq
\frac{4}{\pi}\Hyp_D(z_1,z_2)$.
\end{thm}
The author discussed  the results of these types with
SH. Chen,   S. Ponnusamy and  X. Wang, see also \cite{chma}.
\begin{proof}
If  $D$    is the unit disk  $\mathbb{D}$ this result follows  from  (\ref{eq:dhyp-disk1a}) and it   has been proved   by  Kalaj and Vuorinen  \cite{kavu}.

 In general case one can use a cover $\mathcal{P}:
\mathbb{D}\rightarrow D$ and define $\hat{v}= v\circ \mathcal{P}$.

For  $z,w \in D$, let     $z' \in  \mathcal{P}^{-1}(z)$ ,   $w' \in  \mathcal{P}^{-1}(w)$. Then

${\rm Hyp}_D (\hat{v}z', \hat{v}w') \leq \frac{4}{\pi} {\rm  Hyp}_{\mathbb{D}} (z',w')$.  Hence we get Theorem  \ref{th:kavugen}.
\end{proof}

If  $\Pi= \{w: {\rm Re} w>0 \}$, then $ \Hyp_{\Pi}(w)= \frac{1}{{\rm Re} w}$.

There is tightly   connection between harmonic and holomorphic functions.
 A few year ago I had in mind the following result:

 \begin{thm}\label{Schw_har0}
Suppose that  $D$ is a hyperbolic plane domain and $G=S(a,b)=(a,b)\times \mathbb{R}$,   $-\infty < a < b \leq \infty$,
plane domain  and  $f : D \rightarrow G$ is a  complex harmonic on
hyperbolic domain $D$.
Let $z\in D$,    $h\in T_z \mathbb{C}$, $|h|=1$,  and    $df_z(h)= \lambda v$, $\lambda >0$, $p=f(z)$ and $v \in T_{p} \mathbb{C}$.
If the measure of the angle between $v$ and $e_1=e_1(p)\in T_p C$,$p=f(z)$,is  $\alpha$,  then\\
{\rm(I)} $\lambda \cos \alpha \,\Hyp_G (f(z))  \leq   \Hyp_D(z)$.\\
{\rm(II)} If   $f$ is real valued,    then \\
$\lambda  \,\Hyp_G (fz)  \leq  \Hyp_D(z)$.\\
Hence
$\Hyp_G(f(z_1),f(z_2)\leq \Hyp_D(z_1,z_2)$.
\end{thm}
In the case $D=\D $, {\rm(II)}  is proved   for   $G=S(-1,1)$  in   \cite{kavu}, and    for   $G=S(0,\infty)$  in \cite{Mar_hyp}.
\begin{proof} First suppose that $D=\D$. In general case one can use a cover $\mathcal{P}:\mathbb{D}\rightarrow D$ as in the proof of Theorem  \ref{th:kavugen}.
Let $f=u+iv$, $F+U+iV$  be analytic  such that  $\Re f = \Re F$ on $\mathbb{D}$.
Let  $h\in T_z C$, $|h|=1$,  and    $df_z(h)= \lambda \mathbf{v}$, $\lambda >0$ and $\mathbf{v} \in T_{fz} X$.
If the measure of the angle between $v$ and   $e_1\in T_p C$,$p=f(z)$, is  $\alpha$,  then
$\lambda \cos \alpha \leq  \frac{4}{\pi}  \frac{1- |f(z)|^2}{1- |z|^2}$. Note that  $\Re df_z(h)= du_z(h)$  and  $|du_z(h)|\leq |F'(z)|$.

If we choose $h$ such that  $|f'(z)|= |df_z(h)|$, we get    $ \cos \alpha |f'(z)|\leq |F'(z)|$. Since  $\Hyp_G (Fz)|F'(z)|\leq \Hyp_{\D}(z)$
and   $\Hyp_G (f(z))=\Hyp_G (F(z))= \Hyp_G (\rm {Re}f(z))$  Hence, we have (I).
An application of (I) with  $\alpha =0$ yields  (II).

In particular, if  $G=S(-1,1)$ we have
$$|f'(z)| \cos \alpha \leq  \frac{4}{\pi}  \frac{1- |\rm {Re} f(z)|^2}{1- |z|^2} ,$$
and if   $G=S(0,\infty)$ we have
\be
|f'(z)| \cos \alpha \leq  \frac{2 \rm {Re} f(z)}{1- |z|^2}\,.
\ee

\end{proof}
\section{ Kobayashi-Schwarz lemma  -Several variables}\label{se:new}

\bdefe\label{de:kob}
Let $G$ be bounded connected open subset of complex  Banach space,
$p\in G$  and $\mathbf{v}\in T_p G$.  We define  $k_G(p,\mathbf{v})= \inf\{ |\mathbf{h}|\}$,
where infimum is taking over all $\mathbf{h}\in T_0 \mathbb{C}$ for which
there exists a holomorphic function such that  $\phi: \mathbb{U}
\rightarrow G$ such that $\phi(0)=p$  and $d\phi(\mathbf{h})=\mathbf{v}$.
\edefe
We also use the notation $\Kob_G$ instead of  $k_G$.
We call   $\Kob_G$  Kobayashi-Finsler  norm on tangent bandle.
For some particular domains, we can explicitly compute Kobayashi norm of  a tangent vector by the corresponding angle.

We define the distance function on $G$ by integrating the pseudometric $k_G$: for  $z,z_1 \in G$
\begin{equation}\label{eq:integrated}
    \mathrm{Kob}_G(z,z_1) =\inf_\gamma \int_0^1 k_G(\gamma(t),\dot\gamma(t)) \, dt
\end{equation}
where the infimum is taken  over all piecewise paths $\gamma\colon [0,1]\to G$ with $\gamma(0)=z$ and $\gamma(1)=z_1$.

It is convenient to introduce     $\overline{k}_G=2k_G$.
 By  (\ref{eq:hyp0}), Kobayashi pseudometric $\overline{k}_{\D}$ and the Poncare  metric  coincide on $\D$.

For   complex Banach manifold   $X$  and   $Y$  we denote  by   $\mathcal{O}(X,Y)$  the class of all holomorphic mapping from $X$  into    $Y$.
If $\phi \in \mathcal{O}(\D,X)$ and   $f\in \mathcal{O}(X,Y)$,  then  $\phi\circ f  \in \mathcal{O}(\D,Y)$.

We can express   Kobayashi-Schwarz lemma in geometric form:
\begin{thm}\label{th:kob0}
If $a\in X $   and $b=f(a)$, $u\in T_p X$   and  $u_*= f'(a)u$,
then
\begin{equation}
\mathrm{Kob}(b,u_*)|u_*|_e \leq \mathrm{Kob}(a,u) |u|_e\,.
\end{equation}
\end{thm}
Hence
\begin{thm}[Kobayashi-Schwarz   lemma] \label{th_Earle0}
Suppose that $G$ and   $G_1$  are  bounded connected open subset
of complex Banach space  and  $f: G \rightarrow G_1$  is
holomorphic. Then
\begin{equation}\label{ban0}
\mathrm{Kob}_{G_1}(fz,fz_1)\leq  \mathrm{Kob}_G(z,z_1)
\end{equation}
for all
$z,z_1\in G$.
\end{thm}

If $\phi$ is a holomorphic  map  of   $\D$ into $G$, we define $L_G u(p,v)=\sup \{\lambda:
\, \phi(0)=p,\,\, d\phi_0(1)=\lambda v \}$, and   $L_G (p,v)=\sup    L_G
u(p,v) $, where the supremum is taken over all maps
$\phi:\D\rightarrow G$ which are  analytic in
$\D$
with  $\phi(0)=p$.  Note that $L_G(p,v) k_G(p,v)=1$.  By Definition \ref{de:kob},
\be\label{de:kob2}
\mathrm{Kob}_G(p,v)= \frac{1}{L_G (p,v)}.
\ee
If $G$ is the
unit ball,  we write  $L_\phi(p,v)$ instead of  $L_G \phi(p,v)$.

For $u\in T_p\mathbb{ C}^n$  we denote by  $|u|_e$  euclidean norm.

\subsection{A new version of  Schwarz lemma for the unit ball}
Using classical Schwarz lemma for the unit disk in $\mathbb{C}$, one can derive:
\begin{prop}[Schwarz lemma 1-the unit ball]\label{pr:Kob00}
Suppose that  (i)  $f\in\mathcal{O}(\mathbb{B}_n,\mathbb{B}_m)$  and $f(0)=0$.
Then   (a) $|f'(0)|\leq 1 $.\\
(b) If $u\in T_0\mathbb{C}^n$   and  $u_*= f'(0)u$, then  $|u_*|_e  \leq |u|_e$.
\end{prop}
\begin{proof}

Take an arbitrary  $a \in \mathbb{B}_n$ and set $b=f(a)$.
For $z\in \mathbb{U}$ define  $g(z)=<f(z a^*),b^*>$.  Since $g\in {\rm  Hol}(\mathbb{U},\mathbb{U})$,  then  by the unit disk version  of Schwarz lemma,  we find
(i)   $|g (z)|\leq  |z|$, $z\in \mathbb{U}$.  Choose $z_0$   such that  $a=z_0a^* $
An application of (i)  to $z_0$ ,  yields  $|f (a)|\leq  |a|$. It is straightforward that we get  (a) and (b).
\end{proof}
We need  some properties of bi -holomorphic automorphisms  of unit ball  (see \cite{rud2}  for more details).
For a fixed $z$,    $B_z=\{w: (w-z,z)= 0, |w|^2<1 \}$  and denote by $R(z)$   radius of ball $B_z$.
Denote   by  $P_a(z)$ the orthogonal  projection    onto the subspace $[a]$ generated by $a$ and let  $Q_a=I-P_a$
be  the  projection on  the orthogonal complement.
For $z,a \in B^n $     we define
\begin{equation}
\tilde{z}=\varphi_a(z)= \frac{a-Pz -s_a Qz}{1-(z,a)},
\end{equation}
where  $\displaystyle{P_a(z)=\frac{<z,a>}{<a,a>}a}$    and   $s_a= (1-|a|^2)^{1/2}$.
 Set  $U^a= [a] \cap \mathbb{B}$, $Q^b=b+ [a]^\bot \cap \mathbb{B}_n$,  $\varphi_a^1(z)= \frac{a-Pz }{1-(z,a)}$  and   $\varphi_a^2(z)= \frac{-s_a Qz}{1-(z,a)}$.

 Then one can check that\\
(A2) The restriction of   $\varphi_a$   onto  $U^a$ is automorphisam of     $U^a$ and the restriction onto  $B_z$ maps  it bi-holomorphically mapping  onto   $B_{\tilde{z}}$.

Let $u\in T_p\mathbb{C}^n$ and $p\in \mathbb{B}_n$.
If $A=d\varphi_p$, set    $|A u|_e = M(p,u) |u|_e$, ie.
 \begin{equation}
 M(p,u)= \frac{|A u|_e}{|u|_e}\, .
 \end{equation}
\begin{prop}\label{pr:Kob0}
If the measure of the angle between $u\in T_p\mathbb{ C}^n$ and $p\in \mathbb{B}_n$ is  $\alpha=\alpha(p,u)$,  then

\begin{equation}\label{eq:Kob0}
M(p,u)= M_B(p,u)=\sqrt{\frac{1}{s_p^4}\cos^2\alpha + \frac{1}{s_p^2}\sin^2\alpha}.
\end{equation}
\end{prop}
\begin{proof}
Set   $A^k=d\varphi^k_p$ and   $u=u_1 +u_2$,  where  $u_1\in T_pU^p $  and $u_2\in T_pQ^p $, and  $u'_k  = A^k(u_k)$, $k=1,2$.
By the classical Schwarz lemma 2-the unit disk,  Proposition \ref{pr:Kob00} (Schwarz lemma 1-the unit ball)  and (A2),  $|u'_1|_e= |u_1|_e / s_p^2$ and     $|u'_2|_e= |u_2|_e / s_p$. Then $u'=A(u)=u'_1 +u'_2$  and $u'_1$ and $u'_2$ are orthogonal.

Hence, since   $|u_1|_e = \cos\alpha |u|_e$,    $|u_2|_e = \sin\alpha |u|_e$    and   $|u'|_e  = \sqrt{|u'_1|^2_e +|u'_2|^2_e }$,  we find  (\ref{eq:Kob0}).
\end{proof}


It is clear   that
\begin{equation}
\frac{1}{s_p} \leq M(p,u)\leq \frac{1}{s_p^2}.
\end{equation}
Suppose that  (i)  $f\in\mathcal{O}(\mathbb{B}_n,\mathbb{B}_m)$, $a\in \mathbb{B}_n$   and $b=f(a)$, $u\in T_p\mathbb{ C}^n$   and  $u_*= f'(a)u$.

If $A=d\varphi_a$,  $B=d\varphi_b$ then by
Proposition \ref{pr:Kob0},
we find $|A u|_e = M(a,u) |u|_e $  and   $|Bu_*|_e= M(b,u_*)|u_*|_e$.
By Schwarz 1-unit ball,   $|Bu_*|_e \leq |A u|_e $.  Hence

\begin{thm}\label{Schw_sv0}
Suppose that  $f\in\mathcal{O}(\mathbb{B}_n,\mathbb{B}_m)$, $a\in \mathbb{B}_n$   and $b=f(a)$, $u\in T_p\mathbb{ C}^n$   and  $u_*= f'(a)u$.
Then
\begin{equation}
 M(b,u_*)|u_*|_e \leq M(a,u) |u|_e\,.
\end{equation}

\end{thm}
In particular, we have

\begin{thm}[Schwarz lemma 2-unit ball,\cite{Kal_sch2,MMSchw_Kob}]\label{Schw_svKal}
Suppose that  $f\in\mathcal{O}(\mathbb{B}_n,\mathbb{B}_m)$, $a\in \mathbb{B}_n$   and $b=f(a)$.

Then $ s_a^2|f'(a)|\leq s_b $, i.e.  $(1- |a|^2)|f'(a)|\leq  \sqrt{1- |f(a)|^2} $.
\end{thm}

\begin{thm} Let $a\in \mathbb{B}_n$  and $v \in T_p\mathbb{C}^n$.
For  $\mathbb{B}_n$,
$\mathrm{Kob}(a,v)=M(a,v)|v|_e$.
\end{thm}

\begin{proof}
Let   $\phi$ be   a holomorphic  map  of   $\D$ into $\mathbb{B}_n$,      $\phi(0)=a$,  $v\in T_a C^n$,  $|v|_e=1 $,
$d\phi_0(1)=\lambda v =v'\}$.

$1^\circ$.  Consider first    the case  $a=0$.

Let  $p$ be  the projection on  $[v]$.  Then  $\phi_1= p\circ \phi$  is a holomorphic  map  of   $\D$ into $U^v$.
By classical Schwarz lemma   $|\phi_1'(0)|\leq 1$  and therefore  $\phi^0(\zeta)=v\zeta$ is extremal.
Hence  $\mathrm{Kob}(0,v)=1$.

$2^\circ$.  If  $a\neq 0$, in general  the part of the projection of  $\phi(\D )$ on $[v]$ can    be out of  $\mathbb{B}_n$.So we can not use the procedure
in $1^\circ$ directly  and we  consider   $\phi_a= \varphi_a \circ\phi $  and set  $v_*=  (\varphi_a \circ\phi)'(0)=(d \varphi_a)_a (v')$.

Hence  $|v_*|_e= M(a,v')$  and   by  $1^\circ$,    $|\lambda|M(a,v')\leq 1$.  Hence  the mapping  $\phi^a=\varphi_a \circ\phi^0 $  is  extremal.
\end{proof}

\subsection{polydisk}
For the polydisk, see  \cite{kob}, p.47,
$$\Kob_{\mathbb{U}^n}(z,w)=\max \{\Kob(z_k,w_k): k=1,\cdots,n \} \,. $$

Let  $p=(c,d)\in\mathbb{ U}^2$, $T=(\varphi_c,\varphi_d)$, $A=dT_p$, and  $u\in T_p\mathbb{ C}^2$.

If the measure of the angle between $u\in T_p\mathbb{ C}^2$ and $z_1$-plane
is  $\alpha=\alpha_u=\alpha(p,u)$  and $u'=A(u)$, one can check that(see below)
\begin{equation}\label{eq:poly1}
|u'|_e=M'(p,u)= M'_{\mathbb{U}^2}(p,u)=\sqrt{\frac{1}{s_c^4}\cos^2\alpha + \frac{1}{s_d^4}\sin^2\alpha}\,.
\end{equation}

If $|c|\geq |d|$, then     $$\frac{1}{s_d^2} \leq M'(p,u)\leq \frac{1}{s_c^2}.$$

(III) Now let $\phi (0)=p$,  $u\in T_p\mathbb{C}^2$, $|u|_e=1$,  and  $d\phi_0(1)=\lambda u$ and consider   $T\circ \phi$. If $v= dT_p (u)$, then  $d(T\circ \phi)_0 (1)=\lambda  v $
and   $|v|_e=M'(p,u)|u|_e$. Hence \\
(A3) if $|u|_e=1$,  $M'(p,u) L(p,u)\leq\sqrt{2}$  and $\Kob_{\mathbb{U}^2}(p,u) \geq  M'(p,u)_{\mathbb{U}^2}  |u|_e/\sqrt{2}$.

It turns out that we need an   improvement of  (A3) in order to compute  Kobayashi-Finsler  norm.

Namely, our computation of  Kobayashi-Finsler  norm   on   $\mathbb{ U}^2$  is based on:

(A4)   If  $v=(v_1, v_2) \in T_0\mathbb{C}^2$   and   $|v_1|\geq |v_2|$, then  $\Kob(p,v)= |v_1|=\cos \alpha |v|_e$.
\begin{proof}
Let   $\phi$ be analytic from $\mathbb{U}$  into $\mathbb{U}^2$  and  $\phi (0)=(0,0)$. Then   $\phi_1$  and $\phi_2$  map
$\mathbb{U}$ intoself. Hence    $|d\phi_0(1)|\leq \sqrt{2}$. If  $v=(v_1, v_2) \in T_0\mathbb{ C}^2$   and   $|v_1|\geq |v_2|$, the map  $z\mapsto (zv_1,zv_2)/|v_1|$    shows that  $L(0,v)=1/ |v_1|$    (note that $L(0,v)=|v|/ |v_1| \leq\sqrt{2}$ if    $|v|_e=1$). By equation (\ref{de:kob2}),   $\Kob(p,v)= |v_1|=\cos \alpha |v|_e$.
\end{proof}

Now we check  (\ref{eq:poly1}) and  that \\
(B1)    $|v_1|_e=\cos \alpha |u|_e/s_c^2 $,$|v_2|_e=\sin \alpha |u|_e/s_c^2 $  , and \\
(B2) $k(p,u)= M(p,u)|u|_e$, where     $M(p,u)= M_{\mathbb{U}^2}(p,u)=\max\{\cos\alpha/s_c^2, \sin\alpha/s_d^2\}$.   In particular,  if   $s_d^2/s_c^2 \geq  \tan \alpha$,   then  $k(p,u)= |v_1|_e=\cos \alpha |u|_e/s_c^2 $.\\
Proof.
Set
$A=(A^1,A^2)=dT_p$   and   $u=u_1 +u_2$,  where  $u_1\in T_p[(c,0)]$  and $u_2\in T_p[(0,d)] $, and  $u'_k  = A^k(u_k)$, $k=1,2$.

By the classical Schwarz  lemma 2-the unit disk,  $|u'_1|_e= |u_1|_e / s_p^2$ and     $|u'_2|_e= |u_2|_e / s_p$. Then $u'=A(u)=u'_1 +u'_2$  and $u'_1$ and $u'_2$ are orthogonal.

Hence, since   $|u_1|_e = \cos\alpha |u|_e$,    $|u_2|_e = \sin\alpha |u|_e$    and   $|u'|_e  = \sqrt{|u'_1|^2_e +|u'_2|^2_e }$,  we find (\ref{eq:poly1}),  (B1) and     (B2).
\hfill  $\Box$

Thus  we get
\begin{prop}\label{pr:poly1} Let  $p=(c,d)\in\mathbb{ U}^2$.  If the measure of the angle between $u\in T_p\mathbb{ C}^2$ and $z_1$-plane
is  $\alpha=\alpha_u=\alpha(p,u)$, then
$k_{\mathbb{U}^2}(p,u)=\max\{\cos\alpha/s_c^2, \sin\alpha/s_d^2\}|u|_e$.
\end{prop}
Using  a similar procedure   as in the proof of Proposition \ref{pr:poly1}  one can derive:
\begin{prop}\label{pr:Kob1}
Let  $D$ and $G$ be planar hyperbolic  domains,  $\Omega=D\times G $,   $p=(c,d)\in \Omega $  and  $u\in T_p\mathbb{C}^2$. Then 
$$\Kob(p,u)= M_{\Omega}(p,u) |u|_e,$$ 
where   
$$M(p,u)= M_{\Omega}(p,u)=\max\{\Hyp_D(c)\cos\alpha, \Hyp_G(d)\sin\alpha\}.$$
\end{prop}
We can restate this result in the form:\\
(a) If   ${\Hyp_D^2(c)\cos^2\alpha \geq \Hyp_G^2(d)\sin^2\alpha}$, then $\Kob(p,u)= |v_1|:=\cos \alpha |u|_e \Hyp_D(c)$.\\
(b) If   ${\Hyp_D^2(c)\cos^2\alpha \leq \Hyp_G^2(d)\sin^2\alpha}$, then $\Kob(p,u)= |v_2|:=\sin \alpha |u|_e \Hyp_D(c)$.

\begin{proof}
Let  $\psi^c$ and  $\psi^d$  be  conformal
mappings  of $D$ and $G$  onto $\mathbb{U}$ such that $\psi^c(c)= \psi^d(d)=0$  respectively.  If   $T=(\psi^c,\psi^d)$,   $A= dT_p$  and  $v=A(u)$,  one can check that
$k_\Omega(p,u)=k_{\mathbb{U}^2}(0,v)$, $2|v_1|_e= \lambda_D(c) |u_1|_e $  and   $2|v_2|_e= \lambda_G(c) |u_2|_e $. Hence
\begin{equation}
|v|_e= M'(p,u)= M'_{\Omega}(p,u)=\sqrt{\Hyp_D^2(c)\cos^2\alpha + \Hyp_G^2(d)\sin^2\alpha} \,.
\end{equation}
 If    $|v_1|\geq |v_2|$, then  $k_\Omega(p,u)=\lambda_D(c) |u_1|_e$.
\end{proof}
Using Theorem \ref{th:kob0},  Propositions \ref{pr:Kob0} and  \ref{pr:Kob1}, we have
\begin{thm}\label{Schw_svB_2}
Suppose that  $f\in\mathcal{O}(\mathbb{B}_2,\Omega)$, $a\in \mathbb{B}_2$   and $b=f(a)$, $u\in T_p\mathbb{ C}^2$   and  $u_*= f'(a)u$.
Then
\begin{equation}
 \Kob_\Omega(b,u_*)= M_{\Omega}(b,u_*) |u_*|_e  \leq   M_{\mathbb{B}_2}(a,u) |u|_e\,,
\end{equation}
 where  $\Kob_\Omega$ is described in  Proposition \ref{pr:Kob1}.
\end{thm}



Set   $G=S(a,b)=(a,b)\times \mathbb{R}$,   $-\infty < a < b \leq \infty$.
Note that  $\Hyp_{S(a,b)}(w)= \Hyp_{S(a,b)}({\rm Re}\, w)$,   $w \in S(a,b)$.

If $b< \infty$,  $A(z)= \frac{z-a}{b-a}$ maps conformally  $S(a,b)$   onto $S(0,1)$ and  $B(z)= z-a$
 maps conformally   $S(a,\infty)$  onto  $S(0,\infty)$.

For a vector    $\mathbf{a}=(a_1,a_2,a_3)\in \mathbb{C}^3$,  we define    ${\rm Re}\,
\mathbf{a}=({\rm Re} a_1,{\rm Re} a_2,{\rm Re}a_3)$  and
${\rm Im} \mathbf{a}=({\rm Im} a_1,{\rm Im} a_2,{\rm Im}a_3)$.
Suppose that   $p=(c,d)\in S(a,b)^2 $  and  $u\in T_p\mathbb{C}^2$.\\
We leave the reader to check  that \\
(C1)  $k_{ S(a,b)^2}(p,u)= k_{S(a,b)^2}({\rm Re}p,u)$. \\
(C2)    $k_{S(a,b)^2}({\rm Re}p,{\rm Re}q) \leq k_{S(a,b)^2}(p,q)$,  $p,q\in S(a,b)^2 $.
\begin{thm}\label{Schw_plur}
Let $\underline{u} :\mathbb{B}_2\rightarrow (a,b)^2$  be a pluriharmonic function. Then

\begin{equation}\label{eq:plur1}
\Kob_{S(a,b)^2}(\underline{u}(z),\underline{u}(w))\leq \Kob_{\mathbb{B}_2}(z,w),\quad z,w \in \mathbb{B}_2\,.
\end{equation}
\end{thm}
\begin{proof}Under the hypothesis,  there is an  analytic function  $f :\mathbb{B}_2\rightarrow \mathbb{S}_0^2$
such that ${\rm Re}\,f= {\rm Re}\,\underline{u}$  on   $\mathbb{B}_2$  and  $f
:\mathbb{B}_2\rightarrow (a,b)^2$.

By  Theorem \ref{th_Earle0}  (Kobayashi-Schwarz   lemma) and  (C2)  we have   (\ref{eq:plur1}).
\end{proof}

Let   $p\in  \mathbb{B}_2$  and  $v\in T_p\mathbb{B}^2$. Note that  $df(v)= d{\rm Re}\,f(v)+i d{\rm Im}\,f(v)$.

If the measure of the angle between $v\in T_p\mathbb{C}^2$ and    $x_1x_2$-plane $\beta=\beta(p,u)$,

$|(d{\rm Re}\,f)_p(v)| =  \cos \beta' |df_p(v)| $, where  $\beta'=\beta(p',v_*)$,  $v_*=df_p(v)$  and $p'=f(p)$.

Let $h :\mathbb{B}_2\rightarrow \mathbb{S}_0^2$ be a  pluriharmonic function.
Let $v'=(d{\rm Re}\,h)_p(v)$,  $v_*=dh_p(v)$  and  $p'=h(p)$.

We leave the reader to check  that \\
(D1)   $k_{ S(a,b)^2}(p',v_*)\leq  \tan \beta' k_{\mathbb{B}^2}(p,v)\lambda(p'_2)/\lambda(p'_1)$.

Outline.  Then there is analytic function  $f :\mathbb{B}_2\rightarrow \mathbb{S}_0^2$
such that ${\rm Re}\,f= {\rm Re}\,h$  on   $\mathbb{B}_2$.
Note that  $\Hyp_{\mathbb{S}_0}(w)= \Hyp_{\mathbb{S}_0}({\rm Re}\, w)$,   $w \in \mathbb{S}_0$.


If $I=(-1,1)$ and $J=(0,\infty)$ we can consider  $I^2$, $J^2$  and
$I\times J$.
In a similar way, we can  extend  this result to   pluriharmonic functions   $u :\mathbb{B}_n\rightarrow (a,b)^m$.


\subsection{Invariant  gradient and   Schwarz lemma}
The mapping $w= e^{iz}$  maps  $\mathbb{H}$ onto   the punctured disk.

The Poincare metric on the upper half-plane induces a metric on the punctured disk  $\mathbb{U}'$

    $${\displaystyle ds^{2}={\frac {4}{|q|^{2}(\log |q|^{2})^{2}}}dq\,d{\overline {q}}}, \quad  q\in \mathbb{U}'\,.$$
Hence
$$\displaystyle{{\rm  Hyp}_{\mathbb{U}'}(z)= \frac {-2}{|z|(\log |z|^{2})}}|dz|= \frac {-1}{|z|(\log |z|)}|dz|, \quad z\in \mathbb{U}'\,. $$

Ahlfors \,\cite{ahl1} \,proved a stronger version of Schwarz's lemma and Ahlfors lemma 1.\\
\begin{thm}{\em (\,$Ahlfors \,\,lemma\,\, 2$\,)}.
Let $f$ be an analytic mapping of $\mathbb{D}$ into a region on
which there is given ultrahyperbolic metric $\rho$. Then
$\rho[f(z)]\, |f'(z)|\,\leq \,\lambda \,\,. $
\end{thm}
Let $G$ be open subset of  $\mathbb{B}$,  $f\in C^2(G)$  and   $a \in G$.
Set  $(\tilde{\triangle}) f(a)= \triangle(f\circ \varphi_a)(0)$. Operator  $\tilde{\triangle}$  comutate with  automorphisms of ball and we call it  invariant Laplacian.

For $\lambda \in \mathbb{C}$, we denote by $X_\lambda$ the space of all functions  $f\in C^2(\mathbb{B})$, which satisfy  $\tilde{\triangle}f= \lambda f$. Elements of
$X_0$ we call   $\mathcal{M}$-harmonic  functions.

Recall:
Suppose that  $D$  and $G$  are  hyperbolic planar  domains and  $f$  is  an analytic mapping of $D$ into $G$. Then
${\rm  Hyp}_G(fz) |f'(z)|\,\leq {\rm  Hyp}_D(z)$,  $z\in D$.

If $G\subset \mathbb{C}^n$   and  $f: G\rightarrow C$,
define  $Df(z)=(D_j f(z),...,D_n f(z))$    and   $\tilde{D} f(a)= D(f\circ \varphi_a)(0)$.  For $h\in T_0\mathbb{ C}^n$, set   $u=(d\varphi_a)_0(h) $. Then
$ s_a^2 |h|  \leq  |u|=|(D\varphi_a)_0(h)|\leq  s_a |h|$. Note that
$(d(f\circ \varphi_a))_0(h)= df_a(u)$.

Set  $|Df(z)|=(\sum_{j=1}^n |D_j f(z)|^2)^{1/2}$.
If f(z)=z,  $z\in \mathbb{B}^n$, then  $|Df(z)|=\sqrt{n}$.
Let  $f$  be  complex-valued function defined on $\mathbb{B}^n$.

For   $a  \in  \mathbb{C}^n$,   define  $g(z)=g_a(z)= f(a_1z,..., a_nz)$. Then    $g'(z)=\sum_{k=1}^n  D_k f(za) a_k =df(a)$.

If $f$  is    an analytic mapping of $\mathbb{B}^n$ into  $G$, then
${\rm  Hyp}_G(fz)  |g'(z)|\leq |a|_e $.

Set  $a_k= D_k f(0)/|Df(0)|$, we find  $|g'(0)|= |Df(0)| = |df_0(a)|$.

Hence   $|f'(0)|\leq  |Df(0)|$. Since there is $a$, $|a|=1$, such that   $|f'(0)|= |df_0(a)|$,
$|f'(0)|=  |Df(0)|$.

There is $u \in T_z C^n$  such that   $|Df(z)|= |df_z(u)|$. If $v= (d \varphi)_z (u)$, then
$(d f\circ  \varphi_z)_0 (v)=df_z(u)$. Hence    $|Df(z)|= |df_z(u)|= |(d f\circ  \varphi_z)_0 (v)|\leq |\tilde{D} f(z)| |v|_e $.
Since  $|v|_e \leq   1/s_z^2$, we find  \\
(A)                 $ s_z^2  |Df(z)|\leq |\tilde{D} f(z)|$.

Set $F=f\circ  \varphi_z $.
There is  $v_0 \in T_0 C^n$   such that  $|DF(z)|= |dF_0(v_0)|$ and set   $u_0= (d \varphi)_0 (v_0)$.

Since $|u_0|_e\leq s_z $,    $dF_0(v_0)= df_0(u_0)$, we find    $|df_0(u_0)|\leq |df_0||u_0|\leq s_z |df_0| $. Thus \\
(B)                  $ |\tilde{D} f(z)|\leq s_z  |Df(z)|$.\\
Hence   \\
(C)                 $ s_z^2  |Df(z)|\leq |\tilde{D} f(z)|\leq s_z  |Df(z)|$.

\begin{thm}\label{DyakBloch}
Let   $G$ be   a hyperbolic plane domain   and  let  $f$  be   an analytic mapping of $\mathbb{B}^n$ into $G$,   $a\in \mathbb{B}^n$,  $b=f(a)$, $u\in T_p\mathbb{ C}^n$   and  $u_*= f'(a)u$.
Then
\begin{equation}
{\rm  Hyp}_G(b)|u_*|_e \leq   M_{\mathbb{B}_2}(a,u) |u|_e\,.
\end{equation}
In particular,\\
\rm{(i)} ${\rm  Hyp}_G(fz)  |f'(z)|\leq 1/s_z^2$,\\
\rm{(ii)}    ${\rm  Hyp}_G(fz) |\tilde{D} f(z)|\,\leq 1$.
\end{thm}
By a version of Schwarz lemma,   (i) holds.  For $z=0$,  since   $|\tilde{D} f(0)|=  |f'(0)|$, (ii) holds.
An application  of   (ii)  to   the function $F=f\circ  \varphi_z $ at $z=0$,  and the definition of  $\tilde{D} f(z)$,
show that  (ii) holds  in general.

If $G$  is  the punctured disk, we get a Dyakonov result  \cite{Dyak_Bloch}:
\begin{prop}
Suppose that  $f\in\mathcal{O}(\mathbb{B}_n,\mathbb{U}')$, $a\in \mathbb{B}_n$,  $b=f(a)$     and   $\rho={\rm  Hyp}_{\mathbb{U}'}$.
Then  $ \rho(b) |f'(a)|\leq  2/s_a^2$, i.e.    $ (1- |a|^2) |f'(a)| \leq 2  |b| \ln \frac{1}{|b|}$.
\end{prop}
{\it Acklowedgment}.
We have  discussed the subject at Belgrade Analysis seminar (in 2016) and in particular in connection with minimal surfaces   with  F. Forstneric   and   get useful   information about the subject  via Forstneric  \cite{forst}. We are indebted  to the members of the seminar and to  professor    F. Forstneric for useful   discussions.


\begin{thebibliography}{MM}
\bibitem{ahl1}
L. Ahlfors,  {\it Conformal invariants}, McGraw-Hill Book Company,  1973.
\bibitem{BKr}D.M. Burns and S.G. Krantz,  {\it  Rigidity of holomorphic mappings and a new Schwarz Lemma
at the boundary}, J. Amer. Math. Soc. 7 (1994), 661-676.
\bibitem{Che} D.  Chelts, {\it A generalized Schwarz lemma at the boundary}, Proceedings of the American
Math. Society 129 (2001), No.11, 3275-3278.
\bibitem{dur} {\sc  P. Duren},    {\it Harmonic   mappings in the plane},
Cambridge Univ.   Press,  2004.
\bibitem{Dyak_Bloch}  K. M. Dyakonov,   {\it  Functions in Bloch-type spaces and their moduli},  Ann. Acad. Sci.
Fenn. Math., 41(2):705–712, 2016, arXiv:1603.08140 [math.CV].
\bibitem{forst}  {\it  Communication with F. Forstneric}, 2016.

\bibitem{kavu}  D. Kalaj and M.  Vuorinen, {\it  On harmonic functions and the Schwarz lemma},  Proc. Amer. Math. Soc. 140 (2012), no. 1, 161-165.

\bibitem{Kal_sch2}   D.  Kalaj, {\it  Schwarz lemma for holomorphic mappings in the unit ball},
arXiv:1504.04823v2  [math.CV]  27 Apr 2015

\bibitem{kob}  S. Kobayashi, {\it  Hyperbolic manifolds and holomorphic mappings},
Marcel-Dekker, New York, 1970.
\bibitem{kob1}   S.  Kobayashi, {\it  Hyperbolic Complex Spaces}, Berlin: Springer Nature, (1998), ISBN 3-540-63534-3, MR 1635983
\bibitem{krantz1}  S. G. Krantz,  {\it  The Kobayashi metric, extremal discs, and biholomorphic mappings},
Complex Variables and Elliptic Equations, Volume 57, 2012 - Issue 1
\bibitem{krantz2}  S.  G. Krantz,  {\it  Pseudoconvexity, Analytic
Discs,  and Invariant Metrics},http://www.math.wustl.edu/~sk/indian.pdf

\bibitem{krantz}   S. G. Krantz,  {\it   The Carath\'{e}odory and
Kobayashi Metrics  and Applications in Complex  Analysis},
arXiv:math/0608772v1  [math.CV]  31 Aug 2006
\bibitem{krantz3} S. G. Krantz, {\it  The Schwarz Lemma at the
Boundary},  September 16, 2010
\bibitem{Mar_hyp}M. Markovic, {\it  On harmonic functions and the hyperbolic metric},
 Indag. Math., 26(1):19-23, 2015.


\bibitem{MMSchw_Kob}  M.  Mateljevi\' c,  {\it Schwarz lemma,     Kobayashi Metrics and  FPT}, preprint   November 2016, to appear in Filomat
\bibitem{P_Mel}  P. Melentijevic,  {\it   Invariant  gradient  in refinements  of  Schwarz lemma  and Harnack   inequalities}, manuscript.

\bibitem{rud2} W. Rudin,  {\it Function Theory in the  Unit Ball of $C^n$},
Springer-Verlag, Berlin Heidelberg New York, 1980.

\bibitem{chma} S. H. Chen,  M. Mateljevi\'c, S. Ponnusamy and  X.
Wang,  {\it Schwarz-Pick lemma, equivalent modulus, integral
means and Bloch constant for real harmonic functions},  to appear.

\bibitem{hhChen}   H.  Chen, {\it The Schwarz-Pick lemma and Julia lemma for real planar harmonic mappings},  Sci. China Math.  November 2013, Volume 56, Issue 11, pp 2327-2334
\bibitem{Hboas}  H. Boas,   {\it  Julius and Julia: Mastering the Art of the Schwarz Lemma} - ESIA,  http://www.esi.ac.at
\bibitem{Oss}  R. Osserman, {\it   From Schwarz to Pick to Ahlfors and beyond},
Notices Amer. Math. Soc.46   (1999) 868-873.
\bibitem{Khal}  A. Khalfallah,  {\it  Old and new invariant pseudo-distances defined by pluriharmonic functions},  Jan 2014,  Complex Analysis and Operator Theory,
https://www.researchgate.net/profile/Adel$_-$Khalfallah2
\end{thebibliography}
\end{document}